\documentclass[11pt]{amsart}
\usepackage{geometry}                
\usepackage{graphicx}
\usepackage{amssymb}
\usepackage{epstopdf}
\usepackage{ulem}

\numberwithin{equation}{section}
\usepackage[colorlinks=true]{hyperref}
\usepackage{enumerate}
\hypersetup{urlcolor=blue, citecolor=blue}

\newtheorem{theorem}{Theorem}[section]

\newtheorem{lemma}[theorem]{Lemma}
\newtheorem{proposition}[theorem]{Proposition}

\theoremstyle{definition}
\newtheorem{remark}{Remark}

\newcommand{\rn}{\mathbb R^n}
\newcommand{\R}{\mathbb R}

\def\I{\mathcal{I}}

\newcommand{\h}{\mathcal{H}^{n-1}}

\def\diam{\mathrm {diam}}

\def\dfrac{\displaystyle\frac}
\def\dint{\displaystyle\int}

\def\ds{d\sigma_{x}}
\def\xm{x_{\rm max}}
\def\r{r_{\rm max}}

\newcommand{\Om}{\Omega}

\newcommand{\sm}{\setminus}

\newcommand{\sq}{\subseteq}
\newcommand{\ov}{\overline}

\newcommand{\vps}{\varepsilon}
\newcommand{\ra}{\rightarrow}

\title[ Weinstock inequality in higher dimensions]{ Weinstock inequality in higher dimensions}

\author[D. Bucur, V. Ferone, C. Nitsch, C. Trombetti]{Dorin Bucur, Vincenzo Ferone, Carlo Nitsch, Cristina Trombetti}
\date{}                                           
\address{\vskip1cm\noindent Dorin Bucur \hfill\break\vskip-.2cm
\noindent Institut Universitaire de France and Laboratoire de Math\'ematiques, CNRS UMR 5127,
Universit\' e Savoie Mont Blanc, Campus Scientifique, 73376 Le-Bourget-Du-Lac, France. \hfill\break\vskip-.2cm
\noindent e-mail: {\tt dorin.bucur@univ-savoie.fr}
\hfill\break\vskip-.2cm
\noindent Vincenzo Ferone, Carlo Nitsch, Cristina Trombetti\hfill\break\vskip-.2cm
\noindent Dipartimento di Matematica e Applicazioni ``R. Caccioppoli'', Universit\`{a}
degli Studi di Napoli ``Federico II'', Complesso Universitario Monte S. Angelo, via Cintia
- 80126 Napoli, Italy. \hfill\break\vskip-.2cm
\noindent e-mail: \tt ferone@unina.it; c.nitsch@unina.it;
cristina@unina.it}


\begin{document}

\maketitle

\begin{abstract}

We prove that the Weinstock inequality for the first nonzero Steklov eigenvalue holds in $\R^n$, for $n\ge 3$, in the class of convex sets with prescribed surface area. The key result is a sharp isoperimetric inequality involving simultanously the surface area, the volume and the boundary momentum of convex sets. As a by product, we also obtain some isoperimetric inequalities for the first Wentzell eigenvalue.

\end{abstract}

\section{Introduction }

Let $\Om\sq \R^n$, $n \ge 2$, be a bounded, connected, open set with Lipschitz boundary. The first non-zero Steklov eigenvalue of $\Om$ is defined by
\begin{equation}\label{bf03}
\sigma(\Om):=\min \left \{ \dfrac{ \dint_\Om|\nabla v|^2dx}{\dint_{\partial \Om} v^2 \ds} : v\in H^1(\Om)\sm \{0\}, \int_{\partial \Om}v\ds  =0 \right \}.
\end{equation}
Any minimizer satisfies 
\begin{equation*}
\begin{cases}
-\Delta u=0&\text{in }\Om\\
\frac{\partial u}{\partial n}=\sigma u&\text{on }\partial\Om,
\end{cases}
\end{equation*}
Weinstock proved in \cite{We-a,We} that if $\Om\sq \R^2$ is simply connected, then
$$\sigma(\Om) P(\Om) \le \sigma(B) P(B),$$
where $P(\Om)$ stands for the perimeter of $\Om$ (in dimension $2$, and for surface area measure in dimension $n$) and $B\sq \R^2$ is a ball. In other words, among all simply connected sets of $\R^2$ with prescribed perimeter, the disc maximises the first Steklov eigenvalue. A similar result in dimension $3$ or higher, fails to be true in general. Namely, one can find an annulus $A_\vps= B_1\sm \ov B_\vps$ ($B_r$ denotes the ball of radius $r$) such that 
$$\sigma(A_\vps) P(A_\vps)^\frac{1}{n-1} > \sigma(B) P(B)^\frac{1}{n-1}.$$
For explanations and a complete overview of the question we refer the reader to the recent survey by Girouard and Polterovich \cite{GP17} (see also the references therein). 

The main result of the paper is to prove that the Weinstock inequality holds in any dimension of the space, provided we restrict to the class of convex sets. Namely, we shall prove that for every bounded convex set $\Om \sq \R^n$ we have
$$\sigma(\Om) P(\Om )^\frac{1}{n-1} \le  \sigma(B) P(B)^\frac{1}{n-1}.$$
In fact, our result is slightly sharper, since we prove that
$$\sigma(\Om) \frac{P(\Om )}{V(\Om) ^\frac{n-2}{n} }\le  \sigma(B)  \frac{P(B )}{V(B) ^\frac{n-2}{n}},$$
where $V(\Om)$ stands for the volume of the set $\Om$. The quantity $\frac{P(\Om )}{V(\Om) ^\frac{n-2}{n} }$ turns out to play a crucial role in the upper bound of the Steklov eigenvalues, as Colbois, El Soufi and Girouard assert in  \cite[Theorem 1.3]{CSG11}.

The key argument of our proof is the following sharp isoperimetric inequality. We prove that for every convex set $\Om \sq \R^n$
\begin{equation}\label{bf01}
\dfrac{\dint_{\partial \Omega} |x|^2 \ds}{\left(\dint_{\partial \Omega} \ds \right)\> (V(\Omega))^{2/n}}  \ge \omega_n ^{-\frac	2n},
\end{equation}
where $\omega_n$ stands for the volume of the unit ball of $\R^n$.   Equality holds (only) on the ball centered at the origin.

The proof of this inequality, which fails to be true in the absence of the convexity assumption (e.g. for a smooth set having a lot of boundary surface area close to  the origin), is not trivial. The inequality involves simultaneously three geometric quantities and recalls the inequality proved by Brock \cite{Br}. Indeed,   for every Lipschitz set  $A \sq \R^n$ Brock proved 
\begin{equation}\label{bf02}
\dfrac{ \dint_{\partial A} |x|^2 \ds} {(V(A))^{\frac{n+1}{n}}} \ge \dfrac{\dint_{\partial B} |x|^2 \ds} {(V(B))^{\frac{n+1}{n}}}
\end{equation}
 and, as a consequence, 
$$\sigma(A) V(A)^\frac 1n\le \sigma(B) V(B)^\frac 1n.$$
Among sets with prescribed volume, the boundary momentum and the perimeter are both minimal on balls. Contrary to \eqref{bf02}, 
 inequality \eqref{bf01} involves a ratio between  these  two geometric quantities which are both minimized by the same set. The refinement \eqref{bf01} of inequality \eqref{bf02}  is precisely the crucial technical point of our approach,   the main source of difficulties being the competition between these terms.

A second application of inequality \eqref{bf01} concerns the first non-zero Wentzell eigenvalue. Let $\beta \ge0$. One defines 
\begin{equation}\label{bf04}
\mu (\Om, \beta) =\min\left\{ \dfrac{ \dint_\Om|\nabla v|^2dx+ \beta \int_{\partial \Om} |\nabla_\tau v|^2 \ds }{\dint_{\partial \Om} v^2 \ds} : v\in H^1(\Om)\sm \{0\}, \int_{\partial \Om}v\ds  =0\right \}.
\end{equation}
Any minimizer satisfies 
\begin{equation*}
\begin{cases}
-\Delta u=0&\text{in }\Om\\
-\beta \Delta_\tau u+ \frac{\partial u}{\partial n}=\sigma u&\text{on }\partial\Om,
\end{cases}
\end{equation*}
Above $\nabla _\tau$ and $\Delta_\tau$ stand for the tangential gradient and for the Laplace-Beltrami operator on $\partial \Om$. Clearly, $\mu(\Om,0)=\sigma(\Om)$ and if $\beta \ra +\infty$ then $\frac 1\beta \mu(\Om, \beta)$ converges to the first non zero eigenvalue of the Laplace-Beltrami operator on $\partial \Om$.

We refer the reader to \cite{DKL16} for an analysis of the Wentzell eigenvalues in relationship with the geometry of the domain. The authors obtain in \cite{DKL16} several upper bounds in terms of the geometry, while they conjecture that the ball maximizes the Wentzell eigenvalue among all smooth sets homeomorphic to the ball and  having prescribed volume. We shall prove that this conjecture is true in the class of convex sets. Moreover, if the volume constraint is replaced by the surface area constraint, we prove that the ball gives the maximal value, provided $\beta$ is small enough (see more details in Remark \ref{bf15}). This is not surprising, as the Wentzell eigenvalue is not homogeneous on scaling. This kind of behaviour has already been observed for Robin-Laplacian eigenvalues in \cite{FK15}.

\section{Main result}

Let $\Omega \in \rn$ be a bounded, open convex set, we consider the following  scale invariant  functional

\begin{equation}
\label{ratio}
\lambda(\Omega)=\dfrac{\dint_{\partial \Omega} |x|^2 \ds}{\left(\dint_{\partial \Omega} \ds \right)\> (V(\Omega))^{2/n}} = \dfrac{W(\Omega)}{P(\Omega) (V(\Omega))^{2/n}}
\end{equation}
where $W(\Omega)= \dint_{\partial \Omega} |x|^2 \ds$ denotes the  boundary momentum   of  $\Om$, $P(\Omega)$ and $ V(\Omega)$ denote respectively the perimeter and the Lebesgue measure of $\Omega$.

The main theorem of this section is the following isoperimetric result:

\begin{theorem}
\label{main}
For every bounded, open convex set $\Omega$  of $\rn$ we have
\begin{equation}
\label{disisop}
\lambda(\Omega) \ge \omega_n^{-2/n}
\end{equation}
where $\omega_n$ is the Lebesgue measure of the unit ball in $\rn$ and equality holds only for balls centered at the origin.
\end{theorem}

When $\Omega$ is a bounded open convex set with $C^2$ boundary we will denote respectively by $H_{\partial \Omega}(x)$ the mean curvature of $\partial \Omega$ and by $\nu(x)$ the outer unit normal at $x \in \partial \Omega$.

In what follows we say that $\Omega \in C^{\infty,+}$ if it is a bounded open convex set with a $C^{\infty}$ boundary having non vanishing Gauss curvature.

In order to prove the above theorem we need to introduce the quantities
\begin{equation}
\label{xmax}
r_{\rm max}(\Omega) = {\rm max} \{|x| : x \in \bar \Omega\}, \quad \xm(\Omega) \in \partial \Omega: |\xm(\Omega)|= \r(\Omega).
\end{equation}
We define the {\sl Excess} of $\Omega$ the following quantity
\begin{equation}
\label{excess}
E(\Omega)= \r(\Omega) -  \frac{W(\Omega)}{n V(\Omega)}.
\end{equation}

Then the proof of the Theorem splits into four steps.

{\bf Step 1. Existence of minimizers}
\begin{proposition}
\label{existence}
There exists a convex set minimizing $\lambda(\cdot)$.
\end{proposition}
\begin{proof}
 Given a convex set $\Om$, one can perform a  translation such that $\int_{\partial \Om} x \ds =0$. In this way, its perimeter and volume do not change, while the momentum of the boundary does not increase.  So we can assume that $(\Om_j)_j$ is a minimizing sequence having the same volume and satisfying $\int_{\partial \Om_j} x \ds =0$. In particular, this implies that the origin has to be inside $\Om_j$, for every $j$.  

By Blaschke selection Theorem 1.8.6 in  \cite{Sc}  it is enough to show that $\Omega_j$ have also equibounded diameter. We can assume that $V(\Omega_j)=\omega_n$ and since any ball $B$ centered in the origin is such that $\lambda(B)=\omega_n^{-\frac{2}{n}}$, we know that
$$\lim_j \lambda(\Omega_j)\le\omega_n^{-\frac{2}{n}},$$
and consequently $$\lim_j \frac{W(\Omega_j)}{P(\Omega_j)}\le 1.$$
If arguing by contradiction we assume that $\lim_j \diam(\Omega_j)=\infty,$ convexity yields easily $\lim_j P(\Omega_j)=\infty.$ 
Thereafter if $B_2$ is the ball of radius $2$ centered in the origin it is enough to observe that $\lim_j \frac{\h(\partial\Omega_j\cap B_2)}{\h(\partial\Omega_j\setminus B_2)}=0,$
and we have
$$\lim_j \frac{W(\Omega_j)}{P(\Omega_j)}\ge \lim_j \frac{4}{1+ \frac{\h(\partial\Omega_j\cap B_2)}{\h(\partial\Omega_j\setminus B_2)}}=4$$
which gives a contradiction.
\end{proof}

\subsection{Step 2. A minimizer cannot have negative Excess}

We want to prove that, when $E(\Omega)<0$, we can deform the set $\Omega$ so that $\lambda$ decreases. We use the notion of shape derivative combined with the so called Inverse Mean Curvature Flow (IMCF) (see \cite{G, HI, Ur}).

First of all we observe that there exist sets with negative Excess.
\begin{remark}
We esplicitely observe that the ellipse  ${\mathcal E}=\{ \frac{x^2}{\varepsilon^2} +{y^2}{\varepsilon^2} =1\}$ satisfies, for small $\varepsilon$, $E({\mathcal E})<0$. Indeed $V({\mathcal E})= \pi$, $W({\mathcal E}) = \frac{2}{\varepsilon^2}+o(\varepsilon^2)$, $\r({\mathcal E})  =\frac{1}{\varepsilon}.  $  
\end{remark}

We recall that, given a smooth  field $\theta \in C^{\infty}(\rn,\rn)$, denoting by    $\I$ the identity map and by $\varphi(x)=\theta(x) \cdot \nu(x), \quad x \in \partial\Omega$,  the domain derivative of $\lambda$  in the direction $\theta$ (see \cite{HP05} ) is given by 

\begin{equation}\label{shapeder_0}
\begin{array}{l}
\lambda'(\Omega) = \displaystyle\lim_{h\to 0}\frac{\lambda((\I+h\theta)\Omega)-\lambda(\Omega)}{h}\\\\
= \dint_{\partial\Omega}(n-1) H_{\partial \Omega}(x) \left(|x|^2- \frac{W(\Omega)}{P(\Omega)} \right)\varphi(x) \,\ds + 2 \dint_{\partial\Omega} \left (<x,\nu(x)> - \frac{W(\Omega)}{nV(\Omega)} \right)\varphi(x) \, \ds
\end{array}
\end{equation}

Let now $\Phi_t(x)$ be a smooth family of closed hypersurfaces so that $\Phi_t(\cdot)$ is for all $t\in[0,T]$ a smooth embedding of the sphere $\mathbb{S}^{n-1}$ in $\rn$. The one parameter family $\Phi_t$ is a solution to the inverse mean curvature flow in $[0,T]$ if
$$\frac{\partial}{\partial t}\Phi_t=\frac{1}{H_{\Phi_t}}\nu,$$  for all $t\in[0,T]$.
Namely the family is evolving so that, each point $\Phi_t(x)$ has velocity equal to the reciprocal of the mean curvature of the surface and pointing toward the outer normal direction $\nu$. Local and global existence as well as uniqueness of such a family have been studied in  \cite{G, HI, Ur}   and they are not guaranteed for all initial data $\Phi_0$. However if one consider a bounded convex set $\Omega\in C^{\infty,+}$ and $\Phi_0:\mathbb{S}^{n-1}\to\partial \Omega$ then a solution exists in $[0,\infty]$. Moreover the surface $\Phi_t$ is for all $t>0$ the boundary of a smooth convex set in $C^{\infty,+}$ that we shall denote by $\Omega_t$ and that asymptotically converges to a ball as $t\to\infty$.

In the framework of shape derivative the IMCF correspond to a deformation of the set $\Omega$ such that at each time $t\ge 0$ the function $\varphi(x)$ equals $\frac{1}{H_{\partial \Omega_t}}$ and hence
\begin{equation}\label{shapeder}
\begin{array}{l}
\frac d{dt}\lambda(\Omega_t) = \displaystyle\lim_{h\to 0}\frac{\lambda(\Omega_{t+h})-\lambda(\Omega_t)}{h}\\\\
=2 \dint_{\partial\Omega_t} \left (<x,\nu(x)> - \frac{W(\Omega_t)}{nV(\Omega_t)} \right)\frac{1}{H_{\partial \Omega_t}(x)} \, \ds
\end{array}
\end{equation}
 
\begin{remark}
We observe that if $\Omega$ is a bounded, open convex set of $\rn$
\begin{equation}
\label{mvzero}
\dint_{\partial\Omega}  \left(|x|^2- \frac{W(\Omega)}{P(\Omega)} \right) \,\ds=0;
\end{equation}
\begin{equation}
\label{mvneg}
\dint_{\partial\Omega}  \left (<x,\nu(x)> - \frac{W(\Omega)}{n V(\Omega)} \right) \, \ds
\le 0;
\end{equation}
equality holds if and only is $\Omega$ is a ball centered at the origin. (The inequality is a trivial consequence of Schwartz inequality and of the equality $\dint_{\partial\Omega}  <x,\nu(x)>  \,\ds =nV(\Omega)$).
We also have that
\begin{equation}
\label{pointneg}
| <x,\nu(x)> | - \frac{W(\Omega)}{nV(\Omega)}  \leq E(\Omega) \quad x \in \partial \Omega,
\end{equation}
equality holds if and only is $\Omega$ is a ball centered at the origin.
\end{remark}



\begin{lemma}
\label {isoper}
Let $\Omega$ be a bounded, open convex set of $\rn$ then 
\begin{equation}
\label{totcurv}
\dint_{\partial \Omega} \dfrac{1}{H_{\partial \Omega}(x)} \,\ds \geq \dint_{\partial \Omega^{\sharp}} \dfrac{1}{H_{\partial \Omega^{\sharp}}(x)}  \, \ds
\end{equation}
where $\Omega^{\sharp}$ is a ball such that $V(\Omega)=V(\Omega^\sharp)$.
\end{lemma}

For the proof of the above lemma see Theorem 1 in  \cite{Ros}.

\begin{proposition}
\label{prop1}
Let $\Omega$ be a bounded, open convex set of $\rn$ such that
\begin{equation}
\label{minore}
E(\Omega) < 0,
\end{equation}
then $\Omega$ is not a minimizer of $\lambda(\cdot)$.
\end{proposition}

\begin{proof}
If $\Omega \in C^{\infty,+}$ the proof follows  observing firstly that $\Omega$ is not a ball centered at the origin, otherwise  $E(\Omega)=0$ would be satisfied. Then, choosing $\varphi(x) = \dfrac{1}{H(x)}, x \in \partial \Omega$ in \eqref{shapeder_0},  by \eqref{mvzero}  and  \eqref{pointneg}  we have 

\[\lambda'(\Omega) \leq 2E(\Omega) \dint_{\partial \Omega} \dfrac{1}{H_{\partial \Omega}(x)} \,\ds <0.
\]

If $\Omega$ does not belong to $C^{\infty,+}$ we argue by contradiction and we assume that $\Omega$ minimizes $\lambda(\cdot)$. Let $ (\Omega_k)_k \subset C^{\infty,+}$ such that
$$
\Omega_{k+1} \subset \Omega_k; \quad \lim_{k\rightarrow +\infty}\Omega_k =\Omega \>{\rm in \>the \>Hausdorff \>sense}
$$
obviously

\begin{equation}
\label{limit}
\begin{array}{ll}
&\lim_{k\rightarrow +\infty}V(\Omega_k) =V(\Omega) ; \> \lim_{k\rightarrow +\infty}P(\Omega_k )=P(\Omega);\\ \\
& \lim_{k\rightarrow +\infty}W(\Omega_k )=W(\Omega); \> \lim_{k\rightarrow +\infty}\r(\Omega_k) =\r(\Omega).
\end{array}
\end{equation}

We consider the IMCF for each $\Omega_k$. With $\Omega_k(t), t \ge 0$ we denote the one parameter family generated by the flow and by $\Omega_k(0)=\Omega_k$. By Hadamard formulae  \cite{HP05}

\begin{equation}
\label{derivate1}
\frac d{dt} V(\Omega_k(t))= \dint_{\partial \Omega_k(t)} \dfrac{1}{H_{\partial \Omega_k(t)}(x)} \,\ds; \quad \frac d{dt} P(\Omega_k(t)) = (n-1)P(\Omega_k(t));
\end{equation}

\begin{equation}
\label{derivate2}
 \frac d{dt}\r(\Omega_k(t)) \leq   \dfrac{\r(\Omega_k(t))}{n-1}, 
\end{equation}
and then

\begin{equation}
\label{conclusion}
 \r(\Omega_k(t)) \leq r(\Omega_k) e^{t/(n-1)} \quad t \ge 0.
\end{equation}

The minimality of $\Omega$  implies $\lambda(\Omega) \leq \lambda(\Omega_k(t))$, \eqref{derivate1} together with the monotonicity of the perimeter with respect to the inclusion of convex sets implies   that $P((\Omega_k(t)) \ge P((\Omega_k(0))=P(\Omega_k) >P(\Omega)$, then using 
 \eqref{shapeder}, \eqref{mvzero}, \eqref{pointneg}, \eqref{conclusion} we have for $t \ge 0$:

\begin{equation}
\label{derivate3}
\begin{array}{ll}
&\dfrac d{dt}\lambda(\Omega_k(t)) \leq 2\dint_{\partial \Omega_k(t)} \left( \r(\Omega_k(t))   - \dfrac{W(\Omega_k(t))}{n V(\Omega_k(t))}\right) \dfrac{1}{H_{\partial \Omega_k(t)}(x)} \,\ds
=\\ \\
& =2\dint_{\partial \Omega_k(t)} \left( \r(\Omega_k(t))   - \dfrac{\lambda(\Omega_k(t))P(\Omega_k(t))}{n [V(\Omega_k(t))]^{1-2/n}}\right) \dfrac{1}{H_{\partial \Omega_k(t)}(x)} \,\ds \le \\\\
& \le 2 \left[  \dint_{\partial \Omega_k(t)} \dfrac{1}{H_{\partial \Omega_k(t)}(x)} \,\ds \right] \left(\r(\Omega_k) e^{t/(n-1)}  - \dfrac{\lambda(\Omega)P(\Omega)}{n [V(\Omega_k(t))]^{1-2/n}}\right)=\\\\
&=2V'(\Omega_k(t))  \left(\r(\Omega_k) e^{t/(n-1)}  - \dfrac{\lambda(\Omega)P(\Omega)}{n [V(\Omega_k(t))]^{1-2/n}}\right).\\
\end{array}
\end{equation}

Integrating from $[0,T]$, using \eqref{derivate1} we have $V(\Omega) \le V(\Omega_k) \le V((\Omega_k(t)) < V((\Omega_k(T))$ we have

\[
\lambda(\Omega_k(T))-\lambda(\Omega_k) \le 2 [V(\Omega_k(T)) - V(\Omega_k)] \left(\r(\Omega_k) e^{T/(n-1)}  - \dfrac{\lambda(\Omega)P(\Omega)}{n [V(\Omega_k(T))]^{1-2/n}}\right).
\]

Since \eqref{minore} is in force there exists $\delta >0$ and $T>0$ small enough and  $k_0>>1$ such that
\begin{equation}
\label{small}
\left(\r(\Omega) e^{T/(n-1)}  - \dfrac{\lambda(\Omega)P(\Omega)}{n [V(\Omega)+\delta]^{1-2/n}}\right)<0
\end{equation}

\[
V((\Omega_{k_0}(T)) < V(\Omega) + \delta.
\]

Since the IMCF preserves the inclusion $\Omega_k(T) \subset \Omega_{k_0}(T)$ for $k\ge k_0$ which implies $V((\Omega_{k}(T)) < V(\Omega) + \delta$ for $k\ge k_0$.

Hence for $k\ge k_0$ we obtain

\begin{equation}
\label{final}
\lambda(\Omega_k(T))-\lambda(\Omega_k) \le 2 [V(\Omega_k(T)) - V(\Omega_k)] \left(\r(\Omega_k) e^{T/(n-1)}  - \dfrac{\lambda(\Omega)P(\Omega)}{n [V(\Omega)+\delta]^{1-2/n}}\right)
\end{equation}

By Lemma \ref{isoper} we have 

\[
V'(\Omega_k(t)) \ge n(n-1) V(\Omega_k(t)) \quad {\rm for} \> t \ge 0
\]

and so 

\[
V(\Omega_k(T) \ge V(\Omega_k) e^{n(n-1)T}.
\]

By \eqref{final}, using \eqref{small} and passing into the limit 

\[
[\lim_{k \rightarrow \infty}\lambda(\Omega_k(T))]-\lambda(\Omega) \le 2 V(\Omega)(e^{n(n-1)T}-1) \left(\r(\Omega) e^{T/(n-1)}  - \dfrac{\lambda(\Omega)P(\Omega)}{n [V(\Omega)+\delta]^{1-2/n}}\right)<0
\]

Since  for $k \ge k_0$ we have $\Omega \subset \Omega_k(T) \subset \Omega_{k_0}(T)$, then there exists a convex set $\tilde \Omega$ such that $\lim_{k \rightarrow \infty}\lambda(\Omega_k(T))= \lambda(\tilde \Omega)$,
so
\[
\lambda(\tilde \Omega)<\lambda(\Omega)
\]
which contradicts the minimality of $\Omega$.
\end{proof}

\subsection{ Step 3. A minimizer cannot have positive Excess}

First of all we observe that there exist sets having positive Excess.
\begin{remark}
We esplicitely observe that the ellipse  ${\mathcal E}=\{ \frac{x^2}{(1+\varepsilon)^2} + \frac{y^2}{(1-\varepsilon)^2} =1\}$ satisfies, for small $\varepsilon$, $E({\mathcal E})>0$. Indeed $V({\mathcal E})= \pi(1-\varepsilon^2)$, $W(({\mathcal E})) = 2 \pi +o(\varepsilon)$, $\r(({\mathcal E}))  =1+ \varepsilon $ and 
$\frac{W(\mathcal E)}{nV(\mathcal E)}= 1+ o(\varepsilon).$
\end{remark}

The strategy is to show that when a bounded convex set $\Omega$ satisfies $E(\Omega)>0$ then it is possible to crop the set by cutting with an hyperplane and decrease the value of $\lambda$.

Since ${\dint_{\partial \Omega} x \, \ds}$ is in force we have that $0 \in \Omega$.
For every $\epsilon >0$ we consider an halfspace $T_{\epsilon}$ with outer unit normal pointing in the direction $\xm(\Omega)$ 
and intersecting $\Omega$ 
at a  distance $\epsilon$ from $\xm(\Omega)$. We define $\Omega_{\epsilon} = \Omega \cap T_{\epsilon}$
 and we denote by  $A_{\epsilon}= \partial \Omega_{\epsilon} \cap \partial T_{\epsilon}$.
 
 We set
 \[
 \Delta W = W(\Omega_{\epsilon}) -W(\Omega); \> \Delta V =  V(\Omega_\epsilon) -V(\Omega);  \>\Delta P=P(\Omega_\epsilon) -P(\Omega).
 \]
 Obviously the quantities $ \Delta W$, $\Delta V$ and $\Delta P$ depend on $\epsilon$ and are vanishing as $\epsilon$ goes to $0$

\begin{lemma}
\label{reverse}
There exists a positive constant $C(\Omega)$ such that for all $\epsilon>0$ small enough
\begin{equation}
\label{deltaV}
|\Delta V |  \le C(\Omega) |\Delta P|.
\end{equation} 
Moreover
\begin{equation}
\label{deltaW}
\Delta W \le 2 \r(\Omega)\Delta V +\r^2(\Omega) \Delta P + o(\Delta P) + o(\Delta V).
\end{equation}
\end{lemma}

\begin{proof}
By definition of $\xm(\Omega)$ we have that $\Omega \subseteq B_{\r(\Omega)}$ (the ball centered at the origin with radius $\r(\Omega)$) and $A_{\epsilon} \subseteq  B_{\r(\Omega)} \cap \partial T_{\epsilon}$.
Then
\begin{equation}
\label{diam}
\dfrac{{\rm diam } A_{\epsilon}}{2} \leq \sqrt{2 \r(\Omega) \epsilon}.
\end{equation}

We rotate the coordinate in such a way that the $x_n$ axis  lies in the direction of the outer normal to $T_{\epsilon}$. We denote by $A'_{\epsilon}\subset \R^{n-1}$ the projection of $A_{\epsilon}$ onto the subspace $\{x_n=0\}$ and $g(\cdot):A'_{\epsilon}\to \R$ the concave function describing  $\partial \Omega \setminus \partial \Omega_{\epsilon}$
(see Figure \ref{figura}). By construction $g(0)=\r(\Omega)$. For all $y\in A'_\epsilon$ we set $h(y)= g(y)-(\r(\Omega) -\epsilon)$ and we have
\[
\max h = \epsilon =h(0).
\]

\begin{figure}
\centering
\def\svgwidth{\textwidth}
\begingroup%
  \makeatletter%
  \providecommand\color[2][]{%
    \errmessage{(Inkscape) Color is used for the text in Inkscape, but the package 'color.sty' is not loaded}%
    \renewcommand\color[2][]{}%
  }%
  \providecommand\transparent[1]{%
    \errmessage{(Inkscape) Transparency is used (non-zero) for the text in Inkscape, but the package 'transparent.sty' is not loaded}%
    \renewcommand\transparent[1]{}%
  }%
  \providecommand\rotatebox[2]{#2}%
  \ifx\svgwidth\undefined%
    \setlength{\unitlength}{990.01278815bp}%
    \ifx\svgscale\undefined%
      \relax%
    \else%
      \setlength{\unitlength}{\unitlength * \real{\svgscale}}%
    \fi%
  \else%
    \setlength{\unitlength}{\svgwidth}%
  \fi%
  \global\let\svgwidth\undefined%
  \global\let\svgscale\undefined%
  \makeatother%
  \begin{picture}(1,0.46808394)%
    \put(0,0){\includegraphics[width=\unitlength,page=1]{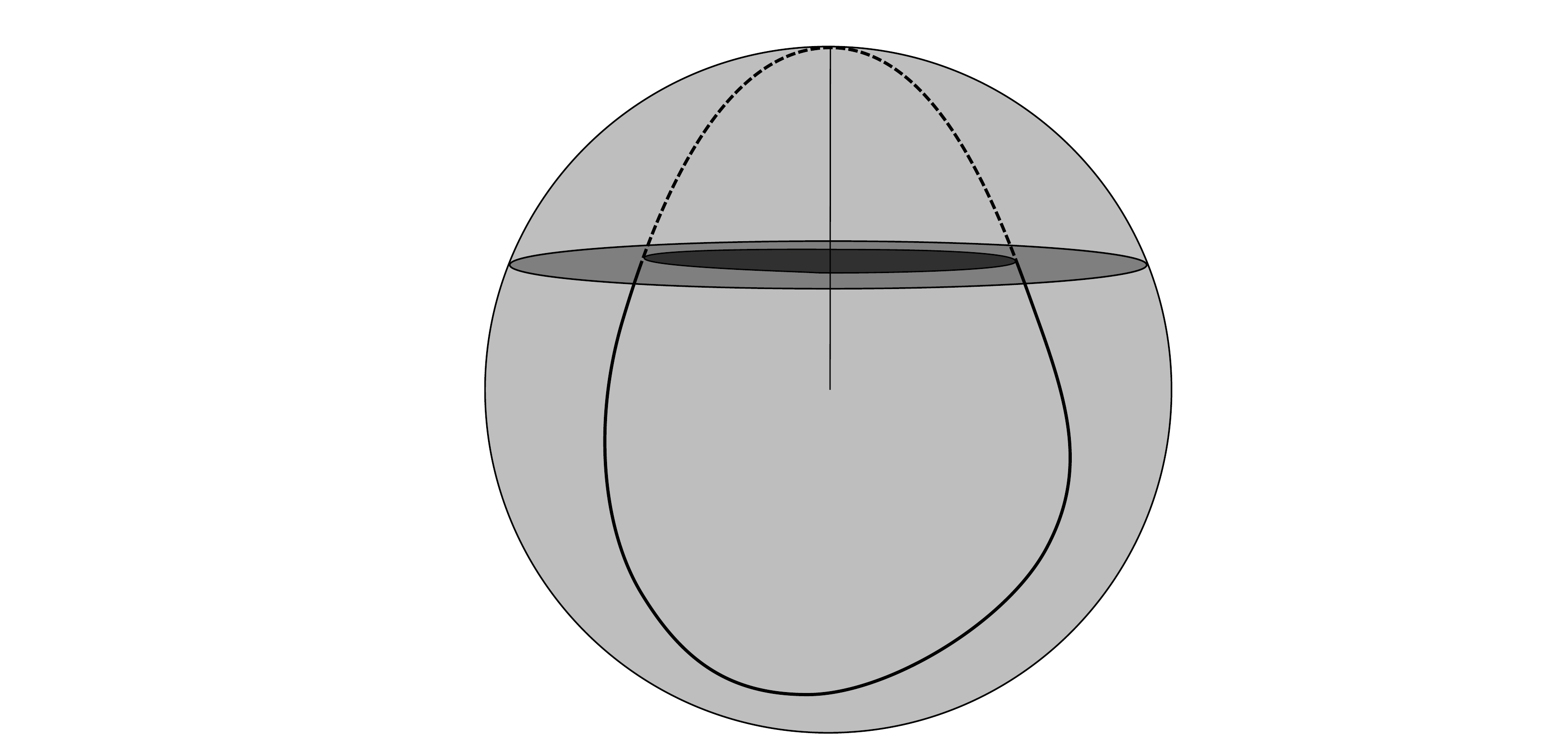}}%
    \put(0.52470852,0.44966561){\color[rgb]{0,0,0}\makebox(0,0)[lb]{\smash{$x_{max}$}}}%
    \put(0.56610041,0.44446968){\color[rgb]{0,0,0}\makebox(0,0)[lb]{\smash{}}}%
    \put(0.49583088,0.36603487){\color[rgb]{0,0,0}\makebox(0,0)[lb]{\smash{$\epsilon$}}}%
    \put(0.63561207,0.35992455){\color[rgb]{0,0,0}\makebox(0,0)[lb]{\smash{$h(y)$}}}%
    \put(0,0){\includegraphics[width=\unitlength,page=2]{h.pdf}}%
    \put(0.56046065,0.25723){\color[rgb]{0,0,0}\makebox(0,0)[lb]{\smash{$A_{\epsilon}$}}}%
    \put(0.50278364,0.20290334){\color[rgb]{0,0,0}\makebox(0,0)[lb]{\smash{$O$}}}%
    \put(0,0){\includegraphics[width=\unitlength,page=3]{h.pdf}}%
    \put(0.09535587,0.20514743){\color[rgb]{0,0,0}\makebox(0,0)[lb]{\smash{$\partial T_\epsilon$}}}%
  \end{picture}%
\endgroup%
\caption{The dashed line corresponds to the portion of the boundary of $\Omega$ which is described by the function $g$.}\label{figura}
\end{figure}

By the concavity of $h$ its graph is above the cone with basis $A'_\epsilon$ and vertex $\epsilon$, then 

\begin{equation}
\label{concavity}
-\Delta V= |\Delta V| = \dint_{A'_{\epsilon}} h(y) \,dy \ge \epsilon \dfrac{{\mathcal L^{n-1}}(A'_\epsilon)}{n}.
\end{equation}

Moreover  \eqref{diam}, \eqref{concavity} and Sobolev Poincar\`e inequality yield
\begin{equation}
\begin{array}{ll}
&|\Delta V| = \dint_{A'_{\epsilon}} h(y) \,dy \le \left (\dint_{A'_{\epsilon}} h(y) \,dy\right)^2 \dfrac {n}{\epsilon {\mathcal L^{n-1}}(A'_\epsilon)} \le \\ \\
& \le C(n) \dfrac{ ({\mathcal L^{n-1}}(A'_\epsilon))^{2/(n-1)}}{ \epsilon} \dint_{A'_{\epsilon}} |D h|^2 \, dy \le C(n) 2 \r(\Omega) (\omega_{n-1})^{2/(n-1)} \dint_{A'_{\epsilon}} |D h|^2.
\end{array}
\end{equation}

Since
\[
-\Delta P=|\Delta P| = \dint_{A'_{\epsilon}} (\sqrt{1+|D g|^2}-1 ) \, dy \ge K(\Omega) \dint_{A'_{\epsilon}} |D g|^2
\]
the proof of \eqref{deltaV} is completed.

On the other hand
\[
-\Delta W= \int_{A'_{\epsilon}}(g^2(y)+|y|^2) \sqrt{1+|Dg|^2} \, dy - \int_{A'_{\epsilon}} [(\r(\Omega) -\epsilon)^2 +|y|^2] \,dy = 
\]
\[
=\int_{A'_{\epsilon}}(g^2(y)-(\r(\Omega) -\epsilon)^2 ) \sqrt{1+|Dg|^2} \, dy + \int_{A'_{\epsilon}} [(\r(\Omega) -\epsilon)^2 +|y|^2] \sqrt{1+|Dg|^2}-1) \,dy \ge
\]
\[
-2 \r(\Omega) \Delta V - \r(\Omega)^2 \Delta P + o(\Delta V) +o(\Delta P) 
\]
and the proof of \eqref{deltaW} is completed.
\end{proof}

\begin{proposition}
\label{prop2}
Let $\Omega$ be a bounded, open convex set of $\rn$ such that
\begin{equation}
\label{maggiore}
E(\Omega) >0,
\end{equation}
then $\Omega$ is not a minimizer of $\lambda(\cdot)$.
\end{proposition}

\begin{proof}
 
By \eqref{deltaW}
 
\begin{equation}
\label{deltaLam}
\begin{array}{ll}
&\Delta \lambda = \dfrac{1}{V(\Omega)^{2/n}P(\Omega)}  \left( \Delta W- \frac{\Delta P }{P(\Omega)} W(\Omega) -\dfrac{2}{n}  \frac{\Delta V }{V(\Omega)}W(\Omega)  \right) + o(\Delta P) + o(\Delta V) =\\ \\
& =  \dfrac{1}{V(\Omega)^{2/n}P(\Omega)}  \left[   2\left( \r(\Omega)   -  \frac{W(\Omega) }{nV(\Omega)} \right)\Delta V    + \left( \r^2(\Omega)-  \frac{W(\Omega)}{P(\Omega)}  \right) \Delta P\right]
+ o(\Delta P) + o(\Delta V) =\\ \\
&=  \dfrac{1}{V(\Omega)^{2/n}P(\Omega)}  \left[   2 \,E(\Omega) \,\Delta V    + \left( \r^2(\Omega)-  \frac{W(\Omega)}{P(\Omega)}  \right) \Delta P\right]
+ o(\Delta P) + o(\Delta V). 
\end{array}
\end{equation}
\smallskip
Observing that $\eqref{maggiore}$ implies that $\Omega$ is not a ball centered at the origin and consequently that $\r^2(\Omega)-  \frac{W(\Omega)}{P(\Omega)}>0$, 
since  $\Delta V <0$ and $\Delta P<0$  we have 
 
\[
 \Delta \lambda <0
 \]
 which concludes the proof.
\end{proof}

\subsection{Step 4. Balls are the unique minimizers having vanishing Excess.}

Again the idea is that one can crop the set $\Omega$ by cutting with an hyperplane to lower $\lambda$. 
\begin{proposition}
\label{prop3}
Let $\Omega$ be a bounded, open convex set of $\rn$ such that
\begin{equation}
\label{uguale}
E(\Omega) = 0,
\end{equation}
then either $\Omega$ is a ball (centered at the origin) or it is not a minimizer of $\lambda(\cdot)$.
\end{proposition}

\begin{proof}
Using  \eqref{deltaLam} \eqref{uguale} and \eqref{deltaV}  we have
\[
\Delta \lambda = \dfrac{1}{V(\Omega)^{2/n}P(\Omega)}  \left[\left( \r^2(\Omega)-  \frac{W(\Omega)}{P(\Omega)}  \right) \right]  \Delta P + o( \Delta P)
\]
\end{proof}
Then either $\r^2(\Omega)=  \frac{W(\Omega)}{P(\Omega)}$, which implies that $\Omega$ is a ball centered at the origin, or $\Delta \lambda <0$ and $\Omega$ is not a minimizer.
This concludes the proof.

\begin{proof}[Proof of Theorem \ref{main}]
The proof follows immediately by Propositions \ref{existence}, \ref{prop1}, \ref{prop2}, \ref{prop3}.
\end{proof}

\begin{remark}
Our inequality \eqref{bf01} is sharp, in the sense that one can not decrease the power of the volume while increasing the power of the perimeter term in \eqref{ratio}. Indeed, let $n=2$ and consider the functionals $\lambda_{\gamma}(\Omega) = \frac{W(\Omega)}{P(\Omega)^{1+\gamma}V(\Omega)^{1-\gamma/2}}$ with $\gamma >0$.
We show that for such functionals balls are {\it  not} minimizers among convex sets. Let ${\mathcal P}_k$ be a regular polygon with $k$ sides, inradius $1$ and let $\alpha = \frac\pi k$.
Then 
\begin{equation}
\begin{array}{ll}
&P({\mathcal P}_k) = 2 \pi \dfrac{ \tan \alpha}{\alpha}; \,V({\mathcal P}_k) = \pi \dfrac{ \tan \alpha}{\alpha}; \, W({\mathcal P}_k) = \dfrac {2\pi}{\alpha}( \tan \alpha+ \dfrac{ \tan^3 \alpha}{3})\\\\
&\lambda_{\gamma}({\mathcal P}_k)= \dfrac{1}{2^{\gamma} \pi^{1+ \gamma/2}} \dfrac{1+\alpha^2/3+o(\alpha^2) }{1+\alpha^2/3(1+\gamma/2)+o(\alpha^2) }.
\end{array}
\end{equation}
So, for $k$ sufficiently large,
\begin{equation*}
\lambda_{\gamma}({\mathcal P}_k) <\dfrac{1}{2^{\gamma} \pi^{1+ \gamma/2}} = \lambda_{\gamma}(B).
\end{equation*}

\end{remark}

\section{Isoperimetric inequalities for Steklov and Wentzell eigenvalues}

 In this section, we show how Theorem \ref{main} leads to new isoperimetric inequalities for the Steklov and Wentzell eigenvalues. 

\begin{theorem}[Weinstock inequality in higher dimensions]
\label{Steklov}
Let $\Omega$ be a bounded, open convex set of $\rn$. Then

\begin{equation}
\label{stek}
\sigma(\Omega) \leq \sigma(\Omega^\star)
\end{equation}
where $\Omega^\star$ is a ball such that $P(\Omega) = P(\Omega^\star)$. Equality holds only if $\Omega$ is a ball.
\end{theorem}

\begin{proof}
Since all the quantities involved are invariant under translations, we can assume that $\partial \Omega$ has the origin as barycenter.
Arguing as in \cite{Br, BDP16, BDPR12}, choosing $x_i$ $i=1, \cdots, n$ as test function in \eqref{bf03} and summing up we have
\[
\sigma(\Omega) \leq \dfrac{nV(\Omega)}{W(\Omega)}
\]
using theorem \ref{main} and the classical isoperimetric inequality we have

\[
\sigma(\Omega) \leq n \omega_n^{2/n} \dfrac{V(\Omega)^{\frac{n-2}{n}}}{P(\Omega)} \le \left[\dfrac{n\omega_n}{P(\Omega)}\right]^{1/(n-1)}= \sigma(\Omega^\star).
\]

\end{proof}

\begin{remark}
The proof above shows in fact the stronger inequality
$$\sigma(\Om) \frac{P(\Om )}{V(\Om) ^\frac{n-2}{n} }\le  \sigma(B)  \frac{P(B )}{V(B) ^\frac{n-2}{n}}.$$
In dimension $n\ge 3$, it was proved in \cite{CSG11} that the $k$-th Steklov eigenvalue satisfies the following bound from above
$$\sigma_k(\Om) \frac{P(\Om )}{V(\Om) ^\frac{n-2}{n} }\le C(n) k^\frac 2n,$$ 
highlighting the idea that  the ratio $\frac{P(\Om )}{V(\Om) ^\frac{n-2}{n} }$ gives sharper upper bounds than the perimeter alone. In the  notation above, the first zero Steklov eigenvalue is denoted by $\sigma_0$.
\end{remark}
\begin{remark}
As in the paper of Brock, one gets for free the following
$$\frac{V(\Om) ^\frac{n-2}{n} }{P(\Om )} \sum_{i=1}^n\frac{1}{\sigma_i(\Om)} \ge \frac{V(B_R) ^\frac{n-2}{n} }{P(B_R )} \sum_{i=1}^n\frac{1}{\sigma_i(B_R)} $$
and, as a consequence,
$$\frac{1}{P(\Om )^{\frac{1}{n-1}}} \sum_{i=1}^ n\frac{1}{\sigma_i(\Om)} \ge  \frac{1}{P(B )^{\frac{1}{n-1}}}\sum_{i=1}^n\frac{1}{\sigma_i(B_R)}.$$
\end{remark}

\begin{theorem}
\label{Wentzell}
Let $\Omega$ be a bounded, open convex set of $\rn$. Then

\begin{equation}
\label{wentz}
\mu(\Omega, \beta) \leq \mu(\Omega^\sharp, \beta)
\end{equation}
where $\Omega^\sharp$ is a ball such that $V(\Omega)=V(\Omega^\sharp)$. Equality holds only if $\Omega$ is a ball.
\end{theorem}

\begin{proof}
Since all the quantities involved are invariant under translations, we can assume that $\partial \Omega$ has the origin as barycenter.
Arguing as in \cite{Br,DKL16}, choosing $x_i$ $i=1, \cdots, n$ as test function in \eqref{bf04} and summing up we have

\begin{equation}\label{bf07}
\mu(\Omega, \beta) \leq \dfrac{ n V(\Omega)+ \beta P(\Omega)}{W(\Omega)}
\end{equation}

using theorem \ref{main} and the classical isoperimetric inequality we have

\[
\mu(\Omega, \beta) \leq \omega_n^{2/n} \dfrac{  n V(\Omega)+ \beta P(\Omega)}{P(\Omega) V(\Omega)^{2/n}} =  \omega_n^{2/n}\left[ n \dfrac{(V(\Omega))^{(1-2/n)}}{P(\Omega)} + \beta V(\Omega)^{-2/n} \right] = \mu(\Omega^\sharp, \beta).
\]

\end{proof}

\begin{remark}\label{bf15}
A natural question is the maximization of $\mu(\Om, \beta)$ under a surface area constraint, precisely to check wether a similar result as Theorem \ref{wentz} holds. Recall that $\mu(\Om, \beta)$ does not behave homogeneously neither in $\beta$ nor to rescalings. Consequently, for a fixed set $\Om$, when $\beta \ra 0_+$ one gets $\mu(\Om, \beta) \ra \sigma(\Om)$. Following Theorem \ref{Steklov} it is reasonable to expect that the ball maximizes the first Wentzell eigenvalue, under surface area constraint and convexity assumption. On the other hand, when $\beta \ra +\infty$, one has that $\frac 1\beta \mu (\Om, \beta)$ converges to the first non-zero Laplace-Beltrami eigenvalue of $\partial \Omega$. The maximum of the first Laplace-Beltrami eigenvalue under surface area constraint is the ball in $\R^3$ in the class of surfaces homeomorphic to the euclidian sphere (so including convex sets) as it was proved by Hersch \cite{He70}. If $n \ge 4$, the maximizer is not, in general, the ball. We refer the reader to \cite{CDS10} for an analysis of this question.  In the class of convex sets, the maximality of the ball under surface area constraint, remains an open question.

\end{remark}
\begin{proposition}
Let $n \ge 3$ and $c >0$ be fixed. There exists $\beta ^*$ depending only on $c$ and $n$ such that for every $\beta \in [0, \beta^*]$ and for every convex set $\Om\sq \R^n$ satisfying $P(\Om)=c$, the following inequality holds true
\begin{equation}\label{bf06}
\mu (\Om, \beta) \le \mu (\Om^\star, \beta),
\end{equation}
where $\Omega^{\star}$ is a ball such that $P(\Omega) = P(\Omega^\star)$.
\end{proposition}
\begin{proof}
Let us consider a convex set $\Om \sq \R^n$ such that $P(\Om)=c$. Assume that inequality \eqref{bf06} fails to be true, i.e.
$$\mu (\Om^\star, \beta)< \mu (\Om, \beta).$$
Following the same argument as in \eqref{bf07}, we get
\begin{equation}\label{bf09}
\dfrac{ nV(\Omega^\star)+ \beta P(\Omega^\star)}{W(\Omega^\star )}= \mu(\Omega^\star, \beta) \leq \dfrac{ n V(\Omega)+ \beta P(\Omega)}{W(\Omega)}.
\end{equation}
Using Theorem \ref{main} to remove $W(\Om)$ we get
$$\dfrac{ n V(\Omega^\star)+ \beta P(\Omega^\star)}{W(\Omega^\star )}\le \dfrac{ nV(\Omega)+ \beta P(\Omega)}{W(\Omega^*)}\frac{P(\Om^\star)}{P(\Om)} \Big (\frac{V(\Omega^\star)}{V(\Omega)}\Big )^\frac 2n.$$
Since $P(\Om^\star)=P(\Om)=c$ we get
$$ n V(\Omega^\star)+ \beta c \le [nV(\Om)+\beta c]  \Big (\frac{V(\Omega^\star)}{V(\Omega)}\Big )^\frac 2n,$$
or, by elementary computation and using $V(\Om^\star) >V(\Om)$
$$nV(\Om^\star)^\frac 2nV(\Om )^\frac 2n \le \beta c \frac{V(\Om^\star)^\frac{2}{n}-V(\Om)^\frac{2}{n}}{V(\Om^\star)^\frac{n-2}{n}-V(\Om)^\frac{n-2}{n}}.$$
There exists a constant  $ C(c,n)$ depending  on $c$ and $n$ only, such that the right hand side is bounded from above by $\beta C(c,n)$, hence

\begin{equation}\label{bf12}
nV(\Om^\star)^\frac 2nV(\Om )^\frac 2n \le \beta C(c,n).
\end{equation}
The consequence of this inequality reads as follows: if $\Om$ does not satisfy \eqref{bf06} for some small $\beta$, then the volume of $\Om$ has to be small as well. 

Assume that $V(\Om) \ge m^*$ (the value $m^*$ will be fixed later) and let us set
$$\ov \beta := \frac {nV(\Om^\star)^\frac 2n(m^*)^\frac 2n }{C(c,n)}.$$
Clearly, \eqref{bf12} fails to be true, if $\beta <\ov \beta ^*$. 

It remains now to find a suitable value $m^*$ depending on $c$ and $n$ such that if $V(\Om)\le m^*$ and $P(\Om) =c$, the inequality
$\mu(\Om, \beta)\le \mu(\Om^\star, \beta)$ holds for small $\beta$. Following the natural order between the first Steklov and Wentzell eigenvalues, together with \eqref{bf09}, we have
$$\sigma(\Om^*) \le  \mu(\Om^\star, \beta)\le \dfrac{ n m^*+ \beta c}{W(\Omega)}.$$
We observe first that $\sigma (\Om^\star)$ is independent on $\beta$. Second, we claim that
$$k(c, n):= \inf\{W(\Om) : \Om \mbox{ convex}, P(\Om)=c \} >0.$$
This is trivial, just observing that
$$W(\Om) \ge W(\tilde B),$$
where $\tilde B$ is the ball, centered at the origin, with surface area $\frac c2$. Indeed, the key observation is that $P(\Om \cap \tilde B) \le P(\tilde B)=\frac c2$, hence
$$W(\Om) \ge \frac c2 \tilde r^2,$$
 $\tilde r$ being the radius of $\tilde B$. Finally, we get
\begin{equation}\label{bf14}
\sigma(\Om^*) \le \dfrac{ n m^*+ \beta c}{k(c, n)}.
\end{equation}
We can find $m^*$ depending only on $c$ and $n$ (for instance $m^*= \frac{\sigma(\Om^*)k(c, n)}{2n}$)  and $\tilde \beta$, such that  \eqref{bf14} fails as soon as $\beta \le \tilde \beta$. 

Finally, for this value $m^*$, we choose as $\beta^*:=\min\{\ov \beta, \tilde \beta\}$ and conclude the proof. 
\end{proof}

\section{Appendix}

After completing the proof of Theorem \ref{main},   we have discovered that the technical report \cite{We},  by Weinstock, contains a proof of the bidimensional version of our result. In the article \cite{We-a} issued from this report, published by Weinstock in  {\it Journal of Rational Mechanics and Analysis}, the section containing the result and the proof is completely removed. Maybe one reason to remove it, is the fact that in two dimensions, this result leads to a weaker version of the isoperimetric inequality.  

The proof given by Weinstock is a typical  two dimensional one, taking advantage from the representation of two dimensional convex sets via their support function. This proof seems hardly adaptable to higher dimensions. We have found suitable to give a brief account of it here. 

Let $D$ be an open, smooth, strictly convex set in the plane satisfying $\int_{\partial D } x \,ds = \int_{\partial D } y \,ds=0$.
Denoted by $(\cos \theta, \sin \theta)\,, \theta \in [0, 2 \pi]$,   the component of the outer unit normal at a point $P=(x,y) \in \partial D$  and by $h(\theta)$ the support function of $D$ then the one parameter family of tangents to $D$ is given by
\begin{equation}
\label{parametric1}
x\cos \theta+y\sin \theta -h(\theta)=0 \quad \theta \in [0, 2 \pi].
\end{equation}

Differentiating \eqref{parametric1} with respect to $\theta$ the parametric equations for $\partial D$ are
\begin{equation}\label{parametric2}
\left\{
\begin{array}{l}
x= h(\theta) \cos \theta -h'(\theta)\sin\theta \\
\\
y= h(\theta) \sin \theta +h'(\theta)\cos\theta  
\end{array}
\right.
\end{equation}
and  from these, differentiating again with respect to $\theta$, we get
\begin{equation}
\left\{
\begin{array}{l}
x'(\theta)=-(h(\theta)+h''(\theta))\sin \theta, \\ \\
 y'(\theta)=(h(\theta)+h''(\theta))\cos \theta, \\ \\
 s'(\theta)^2 = x'^2(\theta) +y'^2(\theta) =(h(\theta)+h''(\theta))^2.
\end{array}
\right.
\label{ascissa}
\end{equation}
The convexity of $D$ ensures that 
\begin{equation}
\label{convex}
h(\theta)+h''(\theta)>0
\end{equation}
hence, choosing $s'(\theta) = h(\theta)+h''(\theta)$ and taking into account that $h$ and $h'$ are continuous and periodic functions of period  $2\pi$, by \eqref{ascissa} we have
\begin{equation}
\label{expressions}
\left\{
\begin{array}{l}
L =P(D)= \dint_0^{2\pi} h(\theta)\, d\theta\\\\
A=V(D)= \dfrac12 \dint_0^{2\pi} (h^2(\theta)+h(\theta)h''(\theta))\, d\theta\\\\
J=W(D)
=  \dint_0^{2\pi} \left(h^3(\theta)+\dfrac12 h^2(\theta)h''(\theta)\right)\, d\theta.
\end{array}
\right.
\end{equation}
Setting for $\theta \in [0,2\pi]$
\begin{equation}
\label{convex2}
h(\theta) = \frac {L}{2\pi} + p(\theta) >0
\end{equation}
with 
\[
\dint_0^{2\pi} p(\theta)\, d\theta=0
\]
from \eqref{convex} it results
\begin{equation}
\label{convex3}
\dfrac {L}{2\pi}+p(\theta)+p''(\theta)>0.
\end{equation}

Then by \eqref{expressions}
\[
A =  \dfrac12 \dint_0^{2\pi} \left(\dfrac{L^2}{4 \pi^2} +p^2(\theta)+ p(\theta) p''(\theta)\right)\, d\theta,
\]
and
\[
J=   \dint_0^{2\pi} \left  ( \dfrac{L^3}{8 \pi^3} + \dfrac{3L}{2\pi} p^2(\theta)  +p^3(\theta) + \dfrac {L}{2\pi} p(\theta) p''(\theta) +\dfrac12 p^2(\theta)p''(\theta))   \right)\,d\theta.
\]

By \eqref{convex2} and \eqref{convex3} we have
\begin{align}
\hskip.5cm&\pi J-LA = \pi  \dint_0^{2\pi} \left( \dfrac{L}{\pi} p^2(\theta)  +p^3(\theta) +\dfrac12 p^2(\theta)p''(\theta)) \right) \, d\theta =\label{fine}\\ \notag\\ 
\notag &\quad= \pi  \dint_0^{2\pi} p^2(\theta) \left(  \dfrac{L}{2\pi} + \dfrac12\left(\frac {L}{2\pi} + p(\theta) \right) + \dfrac12 \left(\dfrac {L}{2\pi}+p(\theta)+p''(\theta) \right)\right)\, d\theta \ge \\ \notag \\\notag
 &\quad\ge \dfrac{L}{2} \dint_0^{2\pi} p^2(\theta)\, d\theta \ge 0,
\end{align}
that is, inequality \eqref{disisop} in the case $n=2$.
Because of  \eqref{convex2} and \eqref{convex3}, equality holds in \eqref{fine} only when $p$ is identically zero that is if $D$ is a disc.
Obviously, by the continuity of $L,A,J$ with respect to variations of $D$, inequality \eqref{fine} holds true if we remove the smoothness assumption on $D$.

The convexity assumption is necessary. Indeed  if $D$ is the cardioid in polar coordinates
\[
\rho= 1-\cos \varphi
\]
it results
\[
\pi J-LA  =- \dfrac{4 \pi}{75}<0.
\]

\end{document}